\documentclass[psamsfonts]{amsart}

\usepackage{amscd, amsmath, amsfonts, amssymb, latexsym}
\usepackage[mathscr]{eucal}
\usepackage{color}
\usepackage{hyperref}
\copyrightinfo{2006}{American Mathematical Society}

\newtheorem{theorem}{Theorem}[section]

\newtheorem{lemma}{Lemma}[section]

\newcommand{\Lap}{\mathcal{L}}

\numberwithin{equation}{section}

\begin{document}

\title[Infinitely divisible modified Bessel distribution]{On an infinitely divisible distribution involving modified Bessel functions}

\author[\'A. Baricz]{\'Arp\'ad Baricz}
\address{Department of Economics, Babe\c{s}-Bolyai University, 400591 Cluj-Napoca, Romania and Institute of Applied Mathematics, \'Obuda University, 1034 Budapest, Hungary}
\email{bariczocsi@yahoo.com}

\author[D.K. Prabhu]{Dhivya Prabhu K}
\address{Institute of Applied Mathematics, \'Obuda University, 1034 Budapest, Hungary}
\email{dhivyanslm@gmail.com}


\begin{abstract}
Mourad E.H. Ismail and Kenneth S. Miller in 1982 conjectured that for $\nu>\mu\geq0$ and $b>a>0$ the function
$$x\mapsto G(x;\mu,\nu)=\left(\frac ba\right)^{\mu-\nu}\frac{I_\mu(a\sqrt{x})I_\nu(b\sqrt{x})}{I_\mu(b\sqrt{x})I_\nu(a\sqrt{x})},$$
where $I_\nu$ stands for the modified Bessel function of the first kind, is the Laplace transform of an infinitely
divisible probability distribution. Their conjecture is equivalent to the monotonicity property of some exponential sums over the zeros $j_{\nu,n}$ with respect to the order, where
$j_{\nu,n}$ denotes the $n$th positive zero of the Bessel function of the first kind of order $\nu$. In this paper our aim is to reduce the problem further to the monotonicity, with respect to $\nu$, of
\[Q_\nu(t)=\sum_{n\geq1}j_{\nu,n}^2e^{-t j_{\nu,n}^2},\qquad t>0,\]
and to prove this monotonicity. In this way we show that indeed the above function $x\mapsto G(x;\mu,\nu)$ is the Laplace transform of an infinitely
divisible probability distribution, and the above conjecture is true. The key ingredient in our proof is an integral representation for the derivative of a quotient of modified Bessel functions of the first kind with respect to the order, combined with Bernstein's theorem and a probability tail argument.
\end{abstract}

\maketitle

\section{\bf The conjecture and its equivalent form}

Let $I_\nu$ denote the modified Bessel function of the first kind.  In
\cite[p. 234]{IsmailMiller1982}, Mourad E.H. Ismail and Kenneth S. Miller considered, for $\nu>\mu\geq0$ and $b>a>0,$ the function
\[G(x;\mu,\nu)=\left(\frac ba\right)^{\mu-\nu}\frac{I_\mu(a\sqrt{x})I_\nu(b\sqrt{x})}{I_\mu(b\sqrt{x})I_\nu(a\sqrt{x})},\qquad x\geq0,\]
and conjectured that $x\mapsto G(x;\mu,\nu)$ is the Laplace transform of an infinitely divisible probability distribution on $[0,\infty)$. In the recent paper \cite{BariczEtAl2026} the above quotient has been revisited and the authors investigated the infinite divisibility after multiplying $G(x;\mu,\nu)$ by an additional factor. In one of the results $G(x;\mu,\nu)$ was multiplied by a normalized reciprocal $I_\sigma$ factor, and in another, for $\nu>1/2$ the expression $G(x;\mu,\nu)$ was multiplied by $e^{-b\sqrt{x}}.$ The proofs in \cite{BariczEtAl2026} show that after the extra factor is introduced, the resulting zero sums split into individually positive pieces using monotonicity of $j_{\nu,n}$ or can be dominated by a continuum integral using zero-spacing inequalities.

Now, we first recall the equivalent formulation of the above conjecture in terms of the positive zeros of the Bessel function of the first kind $J_\nu$. For this, we write
$0<j_{\nu,1}<j_{\nu,2}<\ldots$ for the positive zeros of $J_\nu$ and use the notation $\lambda_n(\nu)=j_{\nu,n}^2.$ Moreover, for $t>0$, we define the infinite sum
\[H_\nu(t)=\sum_{n\geq1} e^{-t\lambda_n(\nu)}\]
and for $\beta>\alpha>0$, we consider the following expression
\[\Theta_\nu(t;\alpha,\beta)=H_\nu(\alpha t)-H_\nu(\beta t)=\sum_{n\geq1}\left(e^{-\alpha t\lambda_n(\nu)}-e^{-\beta t\lambda_n(\nu)}\right).\]
Note that the infinite product representation
\[I_\mu(z)=\frac{(z/2)^\mu}{\Gamma(\mu+1)}\prod_{n\geq1}\left(1+\frac{z^2}{j_{\mu,n}^2}\right), \qquad \mu\geq0,\]
clearly implies that for $\Phi(x)=-\log G(x;\mu,\nu)$ we have the identity
\[\Phi'(x)=\sum_{n\geq1}\left(\frac{1}{x+j_{\mu,n}^2/b^2}-\frac{1}{x+j_{\mu,n}^2/a^2}\right)-
\sum_{n\geq1}\left(\frac{1}{x+j_{\nu,n}^2/b^2}-\frac{1}{x+j_{\nu,n}^2/a^2}\right)\]
and in view of
\[\frac{1}{x+c}=\int_0^\infty e^{-xt}e^{-ct}{\rm d}t,\]
we arrive at
\[\Phi'(x)=\int_0^\infty e^{-xt}\left(\Theta_\mu(t;b^{-2},a^{-2})-\Theta_\nu(t;b^{-2},a^{-2})\right){\rm d}t.\]
Since $G(0;\mu,\nu)=1$, the Schoenberg characterization of infinite divisibility says
that $G(x;\mu,\nu)$ is the Laplace transform of an infinitely divisible probability distribution if and only if $\Phi'$ is completely monotonic or equivalently,
$\Phi$ is a Bernstein function vanishing at the origin \cite[Theorem 1]{IsmailMiller1982}. By Bernstein's theorem and
uniqueness of the Laplace transform, the conjecture is therefore equivalent
to
\[\Theta_\mu(t;\alpha,\beta)\geq \Theta_\nu(t;\alpha,\beta)\]
for all $\nu>\mu\geq0,$ $\beta>\alpha>0$ and $t>0.$ This is the equivalent formulation already mentioned by Mourad E.H. Ismail and Kenneth S. Miller in \cite[Section 4]{IsmailMiller1982}.

Now, we are going to remove the two auxiliary parameters $\alpha$ and $\beta$. For this we define
\[Q_\nu(t)=-H_\nu'(t)=\sum_{n\geq1}\lambda_n(\nu)e^{-t\lambda_n(\nu)}, \qquad t>0.\]
We note that for fixed $\nu$ the termwise differentiation is justified locally uniformly on $(0,\infty).$ By using the McMahon asymptotic expansion for large zeros \cite[eq. 10.21.19]{nist} we have that $j_{\nu,n}=\pi n+\mathcal{O}(1)$ and $j_{\nu,n}^2=\pi^2 n^2+\mathcal{O}(n),$ as $n\to\infty.$ Thus, there exist the constants $c_1,c_2>0$ such that $c_1n^2\leq \lambda_n(\nu)\leq c_2n^2$ and therefore if $t$ is in a compact interval $[\delta,T]\subset(0,\infty),$ then we have that
$$\lambda_n(\nu)e^{-t\lambda_n(\nu)}\leq c_2n^2 e^{-c_1\delta n^2},$$
and the majorant is summable. This in turn implies that the differentiated series converges uniformly by the Weierstrass test.

On the other hand we have
\[\Theta_\nu(t;\alpha,\beta)=\int_{\alpha t}^{\beta t}Q_\nu(s){\rm d}s,\]
and from this it follows immediately that $Q_\mu(s)\geq Q_\nu(s),$ where $s>0,$ implies that $\Theta_\mu(t;\alpha,\beta)\geq\Theta_\nu(t;\alpha,\beta).$ Conversely, suppose that the latter inequality holds for all $\beta>\alpha>0$ and $t>0.$ For arbitrary $v>u>0$ we choose $t=1,$ $\alpha=u$ and $\beta=v.$ Thus we have
\[ \int_u^v\left(Q_\mu(s)-Q_\nu(s)\right){\rm d}s\geq0\]
for every interval $v>u>0.$ Hence the integral of $Q_\mu-Q_\nu$ is nonnegative over every compact interval contained in $(0,\infty).$ Since $Q_\mu-Q_\nu$ is continuous, it follows that $Q_\mu(s)-Q_\nu(s)\geq0$ for every $s>0$. Indeed, if the difference were negative at some point, continuity would make it negative on a sufficiently small interval around that point, contradicting the preceding integral inequality.

Thus, the Ismail--Miller conjecture is equivalent to the following theorem.

\begin{theorem}\label{thm:mainQ}
For every $t>0$, the function
\[\nu\mapsto Q_\nu(t)=\sum_{n\geq1} j_{\nu,n}^2e^{-t j_{\nu,n}^2}\]
is decreasing on $[0,\infty)$.
\end{theorem}

\section{\bf Proof of the theorem}

We begin with estimates which will be used repeatedly to justify the term by term differentiation.

\begin{lemma}\label{lem:bounds}
If $M>0,$ then there exist the constants $c_M,C_M>0$ such that, for every
$\nu\in[0,M]$ and every $n\geq1$
\[c_M\cdot n\leq j_{\nu,n}\leq C_M\cdot n, \qquad 0<\partial j_{\nu,n}/\partial \nu\leq\frac{\pi}{2},\]
and consequently we have
\[c_M^2\cdot n^2\leq\lambda_n(\nu)\leq C_M^2\cdot n^2, \qquad 0<\partial\lambda_n(\nu)/\partial \nu\leq \pi\cdot C_M\cdot n.\]
\end{lemma}

\begin{proof}[\bf Proof of Lemma \ref{lem:bounds}]
Recall that for fixed order, McMahon's expansion in particular says that
\[j_{\nu,n}=\pi\left(n+\frac{\nu}{2}-\frac14\right)+O\left(n^{-1}\right)\]
as $n\to\infty.$ Moreover, it is well-known that the positive zeros $j_{\nu,n}$ are increasing functions of $\nu$. Hence, for $0\leq\nu\leq M$ we have $j_{0,n}\leq j_{\nu,n}\leq j_{M,n}.$ Now, applying the above fixed-order asymptotics at the two endpoint orders $0$ and $M$, and enlarging the constants to cover the finitely many remaining indices,
gives the first two-sided estimate.

For the derivative with respect to the order we use the classical formula
\[\frac{\partial j_{\nu,n}}{\partial\nu}=2j_{\nu,n}\int_0^\infty K_0\left(2j_{\nu,n}\sinh s\right)e^{-2\nu s}{\rm d}s,\]
see for example \cite[eq. 10.21.17]{nist}. Now, observe that the integrand is positive, and thus the derivative is also positive. Since $\sinh s\geq s$, $e^{-2\nu s}\leq1$,
and $K_0$ is positive and decreasing on $(0,\infty)$, we obtain that
\[\frac{\partial j_{\nu,n}}{\partial\nu}\leq 2j_{\nu,n}\int_0^\infty K_0(2j_{\nu,n}s){\rm d}s=\int_0^\infty K_0(u){\rm d}u=\frac{\pi}{2}\]
and this in turn implies that
\[\partial\lambda_n(\nu)/\partial\nu=2j_{\nu,n}\partial j_{\nu,n}/\partial\nu\leq\pi j_{\nu,n},\]
and this completes the proof of the last estimation.
\end{proof}

The preceding bounds justify all differentiations of the next series, which we shall need below.

\begin{lemma}\label{lem:differentiation}
Let $M>0$ be fixed. For every $t>0$, the series defining $H_\nu(t)$ and $Q_\nu(t)$ may be
differentiated term by term with respect to $\nu$, uniformly for $\nu\in[0,M]$. In particular, we have
\[-\frac{\partial H_\nu(t)}{\partial\nu}=t\sum_{n\geq1}\frac{\partial\lambda_n(\nu)}{\partial\nu}e^{-t\lambda_n(\nu)}.\]
Moreover, if
\[F_\nu(t):=-\frac{\partial H_\nu(t)}{\partial\nu},\]
then $F_\nu$ is differentiable for $t>0$ and
\[F_\nu'(t)=\frac{\partial Q_\nu(t)}{\partial\nu}=\sum_{n\geq1}\frac{\partial\lambda_n(\nu)}{\partial\nu}\left(1-t\lambda_n(\nu)\right)e^{-t\lambda_n(\nu)}.\]
Finally, for every $x>0$, the series
\[S_\nu(x)=\sum_{n\geq 1}\frac{\partial\lambda_n(\nu)/\partial\nu}{(x+\lambda_n(\nu))^2}\]
converges locally uniformly in $\nu\geq0$.
\end{lemma}

\begin{proof}[\bf Proof of Lemma \ref{lem:differentiation}]
We fix $t>0$. By Lemma~\ref{lem:bounds}, uniformly for $\nu\in[0,M]$ we have
\[\left|\frac{\partial e^{-t\lambda_n(\nu)}}{\partial\nu}\right|
 =t\frac{\partial\lambda_n(\nu)}{\partial\nu}e^{-t\lambda_n(\nu)}\leq C t n e^{-c t n^2},\]
with positive constants $c,C$ depending only on $M$. Observe that the infinite series having the general term the expression on the
right-hand side of the above inequality is convergent, and therefore the Weierstrass test permits the term by term differentiation
of $H_\nu(t)$ with respect to $\nu$ and gives the stated formula for $F_\nu$. Similarly, we have
\[\left|\frac{\partial}{\partial\nu}\left(\lambda_n(\nu)e^{-t\lambda_n(\nu)}\right)\right|=\frac{\partial\lambda_n(\nu)}{\partial\nu}\left|1-t\lambda_n(\nu)\right|
 e^{-t\lambda_n(\nu)}\leq C n(1+t n^2)e^{-c t n^2},\]
and here the majorant on the right-hand side is also summable. Thus $Q_\nu(t)$ may also be differentiated term by term with respect to $\nu$.

Now, to differentiate $F_\nu$ with respect to $t$, we fix a compact interval
$[t_0,t_1]\subset(0,\infty)$ and for $t\in[t_0,t_1]$ we have
\[\left|\frac{{\rm d}}{{\rm d} t}\left(t\frac{\partial\lambda_n(\nu)}{\partial\nu}e^{-t\lambda_n(\nu)}\right) \right|\leq C n(1+t_1n^2)e^{-c t_0n^2},\]
uniformly for $\nu\in[0,M]$. Observe that the above majorant is also summable, and hence $F_\nu$ is differentiable and its derivative is the displayed series. The
same series was obtained for $\partial Q_\nu(t)/\partial\nu$, and thus $F_\nu'(t)=\partial Q_\nu(t)/\partial\nu.$

Finally, observe that
\[0\leq \frac{\partial\lambda_n(\nu)/\partial\nu}{(x+\lambda_n(\nu))^2}\leq\frac{C n}{c^2n^4}=\frac{C^\ast}{n^3},\]
uniformly on compact intervals of orders, which proves the local uniform convergence of $S_\nu(x)$.
\end{proof}

Now, we are going to use the above results in order to prove Theorem \ref{thm:mainQ}. For $x>0$ let $z=\sqrt{x}$ and consider the function
\[q_\nu(z)=\frac{I_{\nu+1}(z)}{I_\nu(z)}.\]
It is well-known that the infinite product representation of $I_\nu$, followed by logarithmic differentiation, implies that
\[\frac{q_\nu(z)}{2z}=\frac{I_{\nu+1}(\sqrt{x})}{2\sqrt{x}I_\nu(\sqrt{x})}=\sum_{n\geq1}\frac{1}{x+\lambda_n(\nu)}.\]
In view of Lemma \ref{lem:differentiation}, we can differentiate this series with respect to $\nu$ and thus we obtain that
\[S_\nu(x)=-\frac{\partial}{\partial \nu}\left(\frac{q_\nu(z)}{2z}\right).\]
Now, observe that if $p_\nu(z)=-\partial q_\nu(z)/\partial\nu,$ then we clearly have $S_\nu(x)={p_\nu(z)}/(2z).$ On the other hand, in view of the recurrence relations
\cite[eq. 10.29.2]{nist}
$$zI_\nu'(z)=\nu I_\nu(z)+zI_{\nu+1}(z)=zI_{\nu-1}(z)-\nu I_\nu(z),$$
the quotient $q_\nu$ satisfies the following Riccati differential equation
\[q_\nu'(z)=1-q_\nu^2(z)-\frac{2\nu+1}{z}q_\nu(z).\]
Differentiating with respect to $\nu$ both sides of the above equation, yields
\begin{equation}\label{eqpnu}p_\nu'(z)+\left(2q_\nu(z)+\frac{2\nu+1}{z}\right)p_\nu(z) =\frac{2q_\nu(z)}{z}.\end{equation}
In the theory of ordinary differential equations it is well-known that for a first order linear differential equation of the form
$$y'(z)+A(z)y(z)=B(z)$$ an integrating factor is $$\exp\left(\int A(z){\rm d}z\right).$$ Since
\[q_\nu(z)=\frac{I_\nu'(z)}{I_\nu(z)}-\frac{\nu}{z},\]
the integrating factor for the above first order linear differential equation of $p_{\nu}(z),$ that is \eqref{eqpnu}, is $\exp\left(2\log I_{\nu}(z)+\log z\right)=zI_\nu^2(z)$, and therefore multiplying both sides of the equation \eqref{eqpnu} by $zI_\nu^2(z),$ turns its left-hand side into an exact derivative
$$zI_\nu^2(z)p_{\nu}'(z)+\left(2zI_{\nu}(z)I_{\nu}'(z)+I_{\nu}^2(z)\right)p_{\nu}(z)=\frac{{\rm d}}{{\rm d} z}\left(zI_\nu^2(z)p_\nu(z)\right),$$
while the right-hand side becomes
$$zI_\nu^2(z)\frac{2q_\nu(z)}{z}=2I_\nu^2(z)q_\nu(z)=2I_\nu(z)I_{\nu+1}(z).$$
Thus, we obtain
\[\frac{{\rm d}}{{\rm d} z}\left(zI_\nu^2(z)p_\nu(z)\right)=2I_\nu(z)I_{\nu+1}(z).\]
Since for small $z$ we have the expansion
$$I_{\nu}(z)=\frac{(z/2)^{\nu}}{\Gamma(\nu+1)}\left(1+\frac{z^2}{4(\nu+1)}+\mathcal{O}(z^4)\right),$$
we clearly have
\[q_\nu(z)=\frac{z}{2(\nu+1)}+\mathcal{O}(z^3).\]
Thus for small $z$ we arrive at
\[p_\nu(z)=\frac{z}{2(\nu+1)^2}+\mathcal{O}(z^3),\]
and in particlular $p_\nu(z)=\mathcal{O}(z)$ as $z\to 0.$ Hence the integration
constant is zero and we arrive at
\[p_\nu(z)=\frac{2}{zI_\nu^2(z)}\int_0^z I_\nu(s)I_{\nu+1}(s){\rm d}s.\]
Now, we consider the function, defined by $r_\nu(x)=1-2xS_\nu(x).$ Since $x=z^2$ and $S_\nu(x)=p_\nu(z)/(2z)$ we clearly have that
\[r_\nu(x)=1-zp_\nu(z)=1-\frac{2}{I_\nu^2(z)}\int_0^z I_\nu(s)I_{\nu+1}(s){\rm d}s.\]
On the other hand for $\nu>0$ we have $I_\nu(0)=0$, and the recurrence relation
\[I_\nu'(s)=I_{\nu+1}(s)+\frac{\nu}{s}I_\nu(s)\]
implies that
\[\frac{{\rm d}}{{\rm d} s}I_\nu^2(s)=2I_\nu(s)I_{\nu+1}(s)+\frac{2\nu}{s}I_\nu^2(s).\]
Integration of both sides of the above equation from $0$ to $z$ therefore yields
\[r_\nu(x)=\frac{2\nu}{I_\nu^2(\sqrt{x})}\int_0^{\sqrt{x}}\frac{I_\nu^2(s)}{s}{\rm d}s\]
for every $\nu>0,$ and after the substitution $s=r\sqrt{x}$ this becomes
\[r_\nu(x)=2\nu\int_0^1 r^{2\nu-1}\left(r^{-\nu}\frac{I_\nu(r\sqrt{x})}{I_\nu(\sqrt{x})}\right)^2{\rm d}r.\]

Now, we are going to prove that $r_\nu$ is completely monotonic. For $0<r<1$ we define
\[P_{\nu,r}(x)=r^{-\nu}\frac{I_\nu(r\sqrt{x})}{I_\nu(\sqrt{x})},\]
which in view of the infinite product representation of the modified Bessel functions of the first kind can be written as
\[ P_{\nu,r}(x)=\prod_{n\geq1}\frac{\lambda_n(\nu)+r^2x}{\lambda_n(\nu)+x}.\]
On the other hand, for each $n$ the function
\[x\mapsto\frac{\lambda_n(\nu)+r^2x}{\lambda_n(\nu)+x} =r^2+(1-r^2)\frac{\lambda_n(\nu)}{\lambda_n(\nu)+x}\]
is completely monotonic on $(0,\infty).$ Since, an infinite product of the form $\prod_{n\geq1}f_n(x)$ is defined as the pointwise limit of its partial products $\prod_{n=1}^Nf_n(x)$ and if each partial product is completely monotonic, then the infinite product is also completely monotonic. In other words, the set of completely monotonic functions is closed under pointwise convergence. Thus $P_{\nu,r}$ is completely monotonic\footnote{Alternatively, we can see easily that the infinite product representation implies
\[-\frac{{\rm d}}{{\rm d}x}\log P_{\nu,r}(x)=\sum_{n\geq1}\left(\frac1{x+\lambda_n(\nu)}-\frac1{x+\lambda_n(\nu)/r^2}\right).\]
Since every summand is completely monotonic, we obtain that $-\log P_{\nu,r}$ is a Bernstein function vanishing at the origin, so $P_{\nu,r}$ itself is the Laplace transform of an infinitely divisible probability distribution. In particular, $P_{\nu,r}$ is completely monotonic.} and $P_{\nu,r}(0)=1,$ and consequently $P_{\nu,r}^2$ is also completely monotonic. This shows that $r_\nu$ is completely monotonic.  Moreover, we have
\[r_\nu(0^+)=2\nu\int_0^1r^{2\nu-1}{\rm d}r=1\]
and the Bernstein-Widder theorem therefore implies that there exists a unique probability measure $\omega_\nu$ on
$[0,\infty)$ such that
\[r_\nu(x)=\int_0^\infty e^{-xs}{\rm d}\omega_\nu(s).\]

Now, we return to the function $F_\nu,$ which we studied in Lemma \ref{lem:differentiation}, that is
\[F_\nu(t)=t\sum_{n\geq 1}\frac{\partial\lambda_n(\nu)}{\partial\nu}e^{-t\lambda_n(\nu)}.\]
The function $t\mapsto F_\nu(t)$ is nonnegative and continuous for $t>0$, and in addition in view of Lemma \ref{lem:bounds} it is locally
integrable at the origin, that is
\[F_\nu(t)\leq Ct\sum_{n\geq1} n e^{-ctn^2}=\mathcal{O}(1)\]
as $t\to0.$ Since every summand is nonnegative, Tonelli's theorem implies that for $x>0$ the Laplace transform of $F_\nu(x)$, denoted by $\Lap[F_\nu](x)$, is exactly
\[\int_0^\infty e^{-xt}F_\nu(t){\rm d}t=\sum_{n\geq1}\frac{\partial\lambda_n(\nu)}{\partial\nu}\int_0^\infty t e^{-(x+\lambda_n(\nu))t}{\rm d}t
=\sum_{n\geq 1}\frac{\partial\lambda_n(\nu)/\partial\nu}{(x+\lambda_n(\nu))^2}\]
and therefore $2\Lap[F_\nu](x)=2S_\nu(x)=(1-r_\nu(x))/x.$ On the other hand, by using the probability measure $\omega_\nu$, we obtain that the expression $(1-r_\nu(x))/x$ can be rewritten as\footnote{Here $\omega_{\nu}$ is the probability measure given by the Bernstein-Widder theorem and $\omega_\nu((t,\infty))$ means the measure that $\omega_\nu$ assigns to the interval $(t,\infty)$. So, if $X_\nu$ is a random variable with distribution $\omega_\nu,$ then $\omega_\nu((t,\infty))=P(X_\nu>t).$ It is simply the tail function or survival function of the probability measure $\omega_\nu$.}
$$\int_0^\infty\frac{1-e^{-xs}}{x}{\rm d}\omega_\nu(s)=\int_0^\infty\int_0^s e^{-xt}{\rm d}t{\rm d}\omega_\nu(s)=\int_0^\infty e^{-xt}\omega_\nu((t,\infty)){\rm d}t,$$
where Tonelli's theorem is applicable because the integrand is nonnegative. Thus, we arrive at
\[\Lap[2F_\nu](x)=\Lap\left[t\mapsto \omega_\nu((t,\infty))\right](x).\]
Observe that both functions are locally integrable and their Laplace transforms are finite
for every $x>0$. Consequently, by the uniqueness of the Laplace transform, we obtain that $2F_\nu(t)=\omega_\nu((t,\infty))$ for almost every $t>0$. The function on the right-hand side is decreasing and since $F_\nu$ is continuous, it follows that $F_\nu$ is decreasing on $(0,\infty)$. Indeed, if $0<s<t$ and $F_\nu(s)<F_\nu(t)$, continuity gives small ordered neighborhoods of $s$ and $t$ on which this strict inequality holds and by choosing points from the full measure set on which the preceding identity holds contradicts the monotonicity of the tail $\omega_\nu((t,\infty))$. More precisely, by continuity, the strict inequality $F_\nu(s)<F_\nu(t)$ persists on sufficiently small neighborhoods \(U_s\ni s\) and \(U_t\ni t\), with every point of \(U_s\) smaller than every point of \(U_t\). Since the above equality holds on a set of full measure, we may choose \(u\in U_s\) and \(v\in U_t\) where it holds. Then \(u<v\) but
\[
\omega_\nu((u,\infty))=2F_\nu(u)<2F_\nu(v)=\omega_\nu((v,\infty)),
\]
contradicting the monotonicity of the tail function.

Now, recall that in view of Lemma \ref{lem:differentiation} we have $\partial Q_\nu(t)/\partial\nu=F_\nu'(t)$ and since $F_\nu$ is differentiable and decreasing, we arrive at
$\partial Q_\nu(t)/\partial\nu\leq0,$ $\nu>0,$ $t>0.$ Thus, we have that $Q_\mu(t)\geq Q_\nu(t)$ for all $\nu>\mu>0.$

It remains only to include the endpoint $\mu=0$. For this we fix $t>0$ and $M>0$. For $0\leq\rho\leq M$, Lemma \ref{lem:bounds} implies that
\[j_{\rho,n}^2e^{-t j_{\rho,n}^2}\leq C_M^2n^2e^{-t c_M^2n^2},\]
and the right-hand side of this inequality is summable. Since $j_{\rho,n}\to j_{0,n}$ as
$\rho\to0$, the dominated convergence gives $Q_\rho(t)\to Q_0(t).$ Since $Q_\rho(t)\ge Q_\nu(t)$ for every $\nu>\rho>0,$ letting
$\rho\to0$ implies $Q_0(t)\ge Q_\nu(t)$ and hence $Q_\mu(t)\ge Q_\nu(t)$ for all $\nu>\mu\geq0.$

Finally, combining Theorem \ref{thm:mainQ} with the equivalence established in Section 1 proves the Ismail--Miller conjecture.

\section{\bf Some remarks} We would like to point out the following remarks:

{\bf A.} It is worth to mention here that the identity for $r_\nu$ can be obtained directly\footnote{The idea comes from the approach used in the second part of the proof of \cite[Theorem 2.1]{baricz}, where a similar formula has been obtained for the derivative with respect to the order of the logarithmic derivative of the Bessel function of the first kind.}, without using the Riccati differential equation. Indeed, the logarithmic derivative of both sides of the infinite product representation can be written as
\[\frac{zI_\nu'(z)}{I_\nu(z)}=\nu+2z^2\sum_{n\geq1}\frac{1}{z^2+\lambda_n(\nu)}.\]
By Lemma \ref{lem:differentiation}, the differentiation with respect to $\nu$ is
legitimate, and therefore we obtain that
\[\frac{\partial}{\partial\nu}\left(\frac{zI_\nu'(z)}{I_\nu(z)}\right)=r_\nu(z^2).\]
Now, let us consider the expressions
\[u(r)=I_\nu(zr),\qquad v(r)=\partial I_\nu(zr)/\partial\nu.\]
The modified Bessel differential equation and the derivative with respect to the order of both sides imply that
\[(ru'(r))'-\left(\frac{\nu^2}{r}+z^2r\right)u(r)=0\]
and
\[(rv'(r))'-\left(\frac{\nu^2}{r}+z^2r\right)v(r)=\frac{2\nu}{r}u(r).\]
Multiplying both sides of the second equation by $u(r)$, both sides of the first equation by $v(r)$, and subtracting
we arrive at
\[\frac{{\rm d}}{{\rm d}r}\left[r(u(r)v'(r)-v(r)u'(r))\right]=\frac{2\nu}{r}u^2(r).\]
For $\nu>0$ the behavior at the origin implies that $r(u(r)v'(r)-v(r)u'(r))\to0$ as $r\to0$. Indeed, for $\nu>0$ as $r\to0$ we have
\[u(r)=I_\nu(zr)=\frac{(zr/2)^\nu}{\Gamma(\nu+1)}\left(1+\mathcal{O}(r^2)\right),\]
which implies that
\[\frac{v(r)}{u(r)}=\frac{\partial}{\partial\nu}\log I_\nu(zr)=\log r+\mathcal{O}(1),\]
and therefore
\[\left(\frac{v(r)}{u(r)}\right)'=\frac1r+O(r).\]
This in turn implies that as $r\to0$ we have
\[r\left(u(r)v'(r)-v(r)u'(r)\right)=ru^2(r)\left(\frac{v(r)}{u(r)}\right)'=\mathcal{O}(r^{2\nu}).\]
Now, integrating from $0$ to $1$ and dividing by $I_\nu^2(z)$ yields
\[\frac{\partial}{\partial\nu}\left(\frac{zI_\nu'(z)}{I_\nu(z)}\right)=\frac{2\nu}{I_\nu^2(z)} \int_0^1\frac{I_\nu^2(zr)}{r}{\rm d}r,\]
which is precisely the identity used in the proof.

{\bf B.} The positive integral representation for $r_\nu$ was used only for
$\nu>0$. Its limit as $\nu\to0$ is nevertheless explicit
\[r_0(x)=\frac{1}{I_0^2(\sqrt{x})}.\]
Indeed, when $\nu=0$ the relation $I_0'(s)=I_1(s)$, and the Riccati differential equation imply that
\[\int_0^zI_0(s)I_1(s){\rm d}s=\frac12\left(I_0^2(z)-1\right).\]
The complete monotonicity of $r_0$ also follows directly from the infinite product representation of $I_0$.

{\bf C.} We would like point out that in fact the function
\[S_\nu(x)=\sum_{n\geq 1}\frac{\partial\lambda_n(\nu)/\partial\nu}{(x+\lambda_n(\nu))^2}\]
is a generalized Stieltjes transform of order $2,$ that is
\[S_\nu(x)=\int_0^\infty\frac{{\rm d}\eta_\nu(s)}{(x+s)^2}, \qquad \eta_\nu=\sum_{n\geq1}\frac{\partial\lambda_n(\nu)}{\partial\nu} \delta_{\lambda_n(\nu)}.\]
The measure $\eta_\nu$ is positive because the zeros $j_{\nu,n}$ increase
with the order. This representation is closely related in spirit to the
Stieltjes-transform methods used by Mourad E.H. Ismail and collaborators. The above proof
above, however, needs the stronger positive representation of $r_\nu=1-2xS_\nu$, which is supplied by the normalized quotient
$P_{\nu,r}$ and Bernstein's theorem.

\end{document}